\documentclass[12pt, reqno]{amsart} 
\usepackage[T1]{fontenc}
\usepackage{lmodern}
\usepackage{amsmath,amssymb,amsthm,mathtools,mathrsfs}
\usepackage[a4paper,margin=1in]{geometry}
\usepackage{microtype}
\usepackage{enumitem}
\usepackage[hidelinks]{hyperref}

\numberwithin{equation}{section}

\newtheorem{theorem}{Theorem}[section]

\newtheorem{proposition}[theorem]{Proposition}
\newtheorem{corollary}[theorem]{Corollary}
\theoremstyle{definition}

\newtheorem{conjecture}[theorem]{Conjecture}
\theoremstyle{remark}
\newtheorem{remark}[theorem]{Remark}

\DeclareMathOperator{\Ric}{Ric}

\DeclareMathOperator{\Vol}{Vol}
\DeclareMathOperator{\inj}{inj}

\DeclareMathOperator{\tr}{tr}

\newcommand{\R}{\mathbb{R}}

\newcommand{\CP}{\mathbb{CP}}
\newcommand{\Id}{\mathrm{Id}}
\newcommand{\dd}{\,\mathrm{d}}
\newcommand{\eps}{\varepsilon}

\newcommand{\cR}{\mathcal{R}}

\title{On Positively Curved Einstein Four-Manifolds with Euler Characteristic at Most Three}
\author{Liang Cheng}
\date{}

\thanks{Liang Cheng's research is partially supported by the National Natural Science Foundation of China (Grant No. 12671066).}

\address{School of Mathematics and Statistics \& Hubei Key Laboratory of Mathematical Sciences, Central China Normal University, Wuhan, 430079, P.R.China}

\email{chengliang@ccnu.edu.cn}

\begin{document}
	\maketitle
	
	\begin{abstract}
		It is conjectured that a complete, orientable Einstein four-manifold with
		strictly positive sectional curvature must be isometric to either the round
		four-sphere $S^4$ or the complex projective plane $\mathbb{CP}^2$ with the
		Fubini--Study metric. In this paper, we confirm this conjecture when the
		manifold is homeomorphic to $S^4$ or $\mathbb{CP}^2$.
		More precisely, we show that a complete, orientable Einstein four-manifold
		$M$ with strictly positive sectional curvature and Euler characteristic
		$\chi(M) \le 3$ (equivalently, by Freedman's classification theorem,
		homeomorphic to $S^4$ or $\mathbb{CP}^2$) is isometric to either the round
		$S^4$ or the Fubini--Study $\mathbb{CP}^2$. As an application, we obtain
		that a complete, orientable Einstein four-manifold with strictly positive
		sectional curvature admitting a nontrivial Killing field is isometric to
		either the round $S^4$ or the Fubini--Study $\mathbb{CP}^2$.
		\bigskip
		
	\end{abstract}

	\section{Introduction}
	
   	A Riemannian manifold $(M,g)$ is called Einstein if its Ricci curvature
	satisfies $\mathrm{Ric}_g = \lambda g$ for some constant $\lambda$. A
	fundamental problem in Riemannian geometry is to determine which smooth
	manifolds admit Einstein metrics. In dimension four, the Hitchin--Thorpe
	inequality \cite{Thorpe1969, Hitchin1974} gives a topological obstruction
	to the existence of such metrics, while no topological obstruction is
	known in higher dimensions. It is therefore of great interest to classify
	Einstein four-manifolds under natural geometric assumptions.

 In dimension four, only a few compact
	examples for Einstein manifolds with positive sectional curvature  are known. The only known ones are  the round
	sphere $S^4$, the complex projective plane $\CP^2$ and their
	quotients.  This
	motivates the following folklore conjecture.
	
	\begin{conjecture}\label{conj:main}
		A complete, orientable Einstein four-manifold with strictly positive
		sectional curvature is isometric to either the round
		sphere $S^4$ or $\CP^2$ with the Fubini--Study metric.
	\end{conjecture}

Although Conjecture~\ref{conj:main} remains open, a number of partial
results have been obtained under additional curvature hypotheses.
 Throughout this paper we normalize
the metric so that $\Ric_g=3g$; in this normalization,
$K\ge\varepsilon>0$ implies $K\le3-2\varepsilon$.
Berger~\cite{Berger1961} obtained a classification under
$1/4$-pinching, which applies in particular when $K\ge1/2$.
The sufficient lower bound was subsequently reduced to
$(\sqrt{1249}-23)/40$ by Yang~\cite{Yang}, to
$1-\sqrt{2}/2$ by Costa~\cite{Costa2004}, and to $1/4$ by
Cao and Wu~\cite{CaoWu2014}; see also~\cite{Wu2019}.
For upper curvature bounds, Costa~\cite{Costa2004} proved a
classification under $K\le2$, and Cao and Tran~\cite{CaoTran2018}
relaxed this condition to $K\le(14-\sqrt{19})/4$.
Further improvements of the upper curvature bound were obtained by
Cao and Tran~\cite{CaoTran2022} and Cui and Sun~\cite{CuiSun2021},
who also established refined pinching conditions involving the first
positive eigenvalue of the Laplacian.

Related rigidity results also have been obtained under other curvature
conditions, including nonnegative sectional curvature together with
a nontrivial positive-definite intersection form~\cite{GurskyLeBrun},
positivity of the curvature operator~\cite{Tachibana1974},
nonnegative isotropic curvature~\cite{MicallefWang1993,Brendle2010},
and strict pointwise $1/4$-pinching~\cite{BrendleSchoen2009}.

Instead of imposing extra curvature conditions, we confirm
Conjecture~\ref{conj:main} under the topological condition that the Euler
characteristic satisfies $\chi(M)\le 3$.

	\begin{theorem}\label{thm:main}
		Let $(M^4,g)$ be a complete, orientable Einstein manifold with strictly
		positive sectional curvature.
		If the Euler characteristic of $M$ satisfies $$\chi(M)\le 3,$$ then
		$(M^4,g)$ is isometric either to the unit round sphere $S^4$ or to $\CP^2$
		with the Fubini--Study metric.
	\end{theorem}
	\begin{remark}
	Since $M$ is even-dimensional, orientable and positively curved, Synge's
	theorem \cite{Synge1936} implies that $M$ is simply connected. In
	particular, $b_1(M)=b_3(M)=0$, and hence $\chi(M) = 2 + b_2(M)$.
	Hence, the condition $\chi(M)\le 3$ is
	equivalent to $b_2(M)\le 1$. By Freedman's classification theorem
	\cite{Freedman1982}, 
	this in turn is equivalent to $M$ being homeomorphic to $S^4$ or $\CP^2$.
Thus Theorem~\ref{thm:main} is equivalent to saying that
Conjecture~\ref{conj:main} holds under the additional assumption that $M$
is topologically $S^4$ or $\CP^2$.
	\end{remark}

	 By the theorem of Hsiang
	and Kleiner \cite{HsiangKleiner}, a compact, orientable, positively
	curved four-manifold with a nontrivial Killing vector field satisfies $\chi\le 3$. Hence, by Freedman's classification theorem,
homeomorphic to $S^4$ or $\mathbb{CP}^2$; see also Grove and Wilking \cite{GroveWilking2014} for a
	classification up to equivariant diffeomorphism. Combining this with
	Theorem~\ref{thm:main}, we obtain the following.
	
	\begin{corollary}\label{thm:main_2}
		Let $(M^4,g)$ be a complete, orientable Einstein manifold with strictly
		positive sectional curvature.
		If $M$ admits a nontrivial Killing vector field, then $(M,g)$ is
		isometric either to the unit round sphere $S^4$ or to $\CP^2$ with the
		Fubini--Study metric.
	\end{corollary}

\textbf{Next, we outline the strategy of the proof.}
	We normalize the metric so that $\Ric=3g$, and write $V=\Vol(M,g)$. 
	The Gursky--LeBrun gap theorem
	\cite{GurskyLeBrun} implies that each of $W^{\pm}$ either vanishes or
	satisfies $\int_M|W^{\pm}|^2\dd V\ge 6V$. When $\chi(M)\le 3$, the
	Gauss--Bonnet--Chern and signature formulas bound the relevant Weyl
	component from above, e.g.\ $\int_M|W^-|^2\dd V=6\pi^2-3V$ if
	$(\chi,\tau)=(3,1)$. Hence a nonvanishing component would force
	$V\le 8\pi^2/9$ (if $\chi=2$) or $V\le 2\pi^2/3$ (if $\chi=3$).
	The key  ingredient is an improved volume estimate that rules this out.
	Since $0<K<3$, Klingenberg's estimate gives $\inj(M)>\pi/\sqrt3$. However,
	standard comparison only yields $V>8\pi^2/27$. Instead, we analyze the
	Riccati equation along radial geodesics. The two-sided curvature bounds
	pinch the eigenvalues of the shape operator between
	$\sqrt3\cot(\sqrt3\,t)$ and $1/t$. Combined with the Einstein condition,
	this pinching yields a sharper lower bound for the mean curvature, which
	gives
	\[
	V>\frac{2\pi^2(\pi^2-4)}{9}
	\]
	(Proposition~\ref{prop:candle}). This exceeds both thresholds, so $g$ is
	half-conformally flat. The classification of Hitchin and Friedrich--Kurke
	then shows that $(M,g)$ is the round $S^4$ or $\CP^2$ with the
	Fubini--Study metric.

\textbf{	The organization of the paper as follows.}
The paper is organized as follows. In Section~1 we collect some
preliminaries on Einstein four-manifolds and recall the Gursky--LeBrun gap
theorem. In Section~2 we prove the improved volume estimate. In Section~3
we prove Theorem~\ref{thm:main} and Corollary~\ref{thm:main_2}.

\section{Preliminaries}

	\subsection{Some properties of an 
		Einstein 4-manifold}	
	In this subsection, we recall some properties of an 
	Einstein 4-manifold.
	
	Let \((M^4,g)\) be oriented Einstein 4-manifold.  Denote $\cR$, K, $\Ric$, $s$, $W$ be the 
	Riemann curvature, sectional curvature, Ricci curvature, scalar curvature and Weyl curvature, respectively. 
	The Hodge star operator $*$ induces a natural 
	decomposition of the vector bundle of 2-forms,
	Since \(*^2=1\) on two-forms,
	\[
	\Lambda^2=\Lambda^+\oplus\Lambda^-,
	\qquad
	\Lambda^\pm=\{\omega:*\omega=\pm\omega\},
	\]
	each summand having real rank three.  The curvature operator of \((M^4,g)\) is
	\begin{equation}
		\cR=
		\begin{pmatrix}
			W^+ + \frac{s}{12}\Id & 0\\
			0 & W^-+\frac{s}{12}\Id
		\end{pmatrix},
	\end{equation}
	where $W^{\pm}$ is the restriction of $W$ to $\Lambda^{\pm}$.
	
	Next, we recall some properties of a closed Einstein 4-manifold. First, the formulas for curvature and topology of such a closed 4-manifold obtained via the Gauss-Bonnet-Chern 
	formula of the Euler characteristic $\chi(M)$ and Hirzebruch formula of the topology signature $\tau(M)$
	(see \cite{Besse} ):
	\begin{align}
		8\pi^2\chi(M)
		&=
		\int_M
		\left(
		|W^+|^2+|W^-|^2-\frac12|E|^2+\frac{s^2}{24}
		\right)\dd V,
		\label{eq:CGB-general}\\
		12\pi^2\tau(M)
		&=
		\int_M\left(|W^+|^2-|W^-|^2\right)\dd V.
		\label{eq:signature-general}
	\end{align}
	For an Einstein metric \(E\equiv0\) and \eqref{eq:CGB-general} becomes
	\begin{equation}\label{eq:CGB-einstein}
		8\pi^2\chi(M)
		=
		\int_M\left(|W^+|^2+|W^-|^2+\frac{s^2}{24}\right)\dd V.
	\end{equation}
	For an Einstein four-manifold,    a formula due to
	Derdzi\'nski \cite{Derdzinski} indicates \(W^\pm\) satisfies 
	\begin{equation}\label{eq:Derdzinski}
		\frac12\Delta|W^\pm|^2=|\nabla W^\pm|^2+\frac{s}{2}|W^\pm|^2-18\det W^\pm.
	\end{equation}

	\subsection{The Gursky--LeBrun's  gap theorem}
	\label{sec:gap}

	Next, we need the following gap theorem obtained by Gursky and LeBrun \cite{GurskyLeBrun}.  For the sake of completeness, we briefly provide their proof here.
	
	\begin{theorem}[Gursky--LeBrun's gap theorem]\label{thm:GL-gap}
		Let \((M^4,g)\) be a compact oriented Einstein four-manifold with  scalar curvature \(s>0\).
		Set  \(U=W^+\) or \(U=W^-\). Then 
		either \(U\equiv0\), or
		\begin{equation}\label{eq:gap}
			\int_M|U|^2\,\dd V\ge\frac{Y([g])^2}{24}
			=\int_M\frac{s^2}{24}\,\dd V,
		\end{equation}
		where 	\begin{equation}\label{eq:Yamabe-def}
			Y([g])
			=
			\inf_{0\ne f\in C^\infty(M)}
			\frac{\int_M\left(6|df|^2+sf^2\right)\dd V}
			{\left(\int_M|f|^4\dd V\right)^{1/2}}
		\end{equation}
		is the Yamabe constant.
	\end{theorem}

	\begin{proof}[Proof of Theorem~\ref{thm:GL-gap}]
		
		First, a fundamental result of Obata \cite{Obata} implies that any Einstein metric
		is a Yamabe minimizer. Indeed, notice that	 
		the conformal Laplacian \(L_g=-6\Delta+s\) has positive first
		eigenvalue and \(Y([g])>0\).  Let \(\widehat g=u^2g\) be a Yamabe minimizer,
		with constant scalar curvature \(\widehat s>0\) \cite{Schoen}; then
		\(Y([g])^2=\widehat s^{\,2}\Vol(M,\widehat g)\).  Applying
		\eqref{eq:CGB-general} to \(g\) and to \(\widehat g\) and subtracting, using
		the conformal invariance of \(\int|U|^2\dd V\) in dimension four, yields
		\begin{equation}
			\frac{s^2V}{24}
			=
			\frac{Y([g])^2}{24}
			-
			\frac12\int_M|\Ric^\circ_{\widehat g}|^2\,\dd V_{\widehat g},
		\end{equation}
		so \(Y([g])\ge s\sqrt V\).  Testing \eqref{eq:Yamabe-def} with \(f\equiv1\)
		gives the reverse inequality.  Hence, we have
		\begin{equation}\label{lem:yamabe}
			Y([g])=s\sqrt V, 
		\end{equation}
		with \(V=\Vol(M,g)\).
		
		On \(\{|U|>0\}\) set \(u=|U|\).  From
		\(\frac12\Delta u^2=u\Delta u+|\nabla u|^2\) and \eqref{eq:Derdzinski},
		\[
		u\Delta u
		=	
		|\nabla U|^2-|\nabla u|^2+\frac{s}{2}u^2-18\det U.
		\]
		Furthermore, it was observed by Gursky-LeBrun \cite{GurskyLeBrun} and Yang \cite{Yang} that $W^{\pm}$ satisfies 
		the following refined Kato inequality (proven to be optimal by \cite{Branson2000, CalderbankGauduchonHerzlich2000}),
		\begin{equation}\label{lem:kato}	
			|\nabla U|^2\ge\frac53\,|d|U||^2.
		\end{equation}
		It follows that \(|\nabla U|^2-|\nabla u|^2\ge\frac23|\nabla u|^2\),
		and hence
		\begin{equation}\label{lem:det-bound}
			18\det U\le18\frac{u^3}{3\sqrt6}
			=\sqrt6\,u^3.
		\end{equation} 
		Hence
		\begin{equation}\label{eq:u-diff}
			u\Delta u\ge\frac23|\nabla u|^2+\frac{s}{2}u^2-\sqrt6\,u^3.
		\end{equation}
		For \(v=u^{1/3}\) one has
		\(\Delta v=\frac13u^{-2/3}\Delta u-\frac29u^{-5/3}|\nabla u|^2\); dividing
		\eqref{eq:u-diff} by \(u\) and substituting, the gradient terms cancel
		exactly.  Thus
		\begin{equation}
			(-6\Delta+s)v\le2\sqrt6\,v^4.
		\end{equation}
		
		For
		\(\eps>0\) put \(r_\eps=(|U|^2+\eps^2)^{1/2}\) and \(v_\eps=r_\eps^{1/3}\).
		Differentiating \(r_\eps\) and using \eqref{eq:Derdzinski} gives
		\[
		r_\eps\Delta r_\eps
		=	
		|\nabla U|^2-|\nabla r_\eps|^2+\frac{s}{2}|U|^2-18\det U,
		\]
		while \(|\nabla r_\eps|=\frac{|U|}{r_\eps}|d|U||\le|d|U||\), so that
		\(|\nabla U|^2-\frac53|\nabla r_\eps|^2\ge0\) by \eqref{lem:kato}.
		Repeating the power computation with \(|U|^2=r_\eps^2-\eps^2\) yields
		\begin{equation}
			(-6\Delta+s)v_\eps\le2\sqrt6\,v_\eps^4+s\eps^2r_\eps^{-5/3}.
		\end{equation}
		Multiplying by \(v_\eps\) and integrating by parts,
		\[
		\int_M\left(6|dv_\eps|^2+sv_\eps^2\right)\dd V
		\le
		2\sqrt6\int_Mv_\eps^5\,\dd V
		+
		s\eps^2\int_Mr_\eps^{-4/3}\,\dd V.
		\]
		Since \(r_\eps\ge\eps\), the error is at most \(s\,\Vol(M)\eps^{2/3}\).  Hence
		\(v_\eps\) is bounded in \(H^1\) and, after a subsequence, converges weakly in
		\(H^1\) and strongly in \(L^q\), \(q<4\), to \(v=|U|^{1/3}\); dominated
		convergence applies to \(v_\eps^5\), and lower semicontinuity of the energy
		gives
		\begin{equation}
			\int_M\left(6|dv|^2+sv^2\right)\dd V\le2\sqrt6\int_Mv^5\,\dd V.
		\end{equation}
		The Yamabe inequality, together with \(v^3=|U|\), so that
		\(v^5=v^2|U|\), then gives
		\[
		Y([g])\|v\|_4^2
		\le
		2\sqrt6\int_Mv^2|U|\,\dd V
		\le
		2\sqrt6\,\|v\|_4^2\,\|U\|_2.
		\]
		If \(U\not\equiv0\) then \(\|v\|_4>0\), so \(Y([g])\le2\sqrt6\|U\|_2\), which
		is \eqref{eq:gap}.
	\end{proof}

	\section{The volume estimates}
	\label{sec:volume}

	Let \((M^4,g)\) be a compact, orientable, with \(\Ric=3g\) and \(K>0\). 
	Then \(\max K<3\).
	By Synge's theorem \(M\) is simply connected, so the
	even-dimensional Klingenberg estimate \cite{Klingenberg1959} (cf.\ \cite{Petersen}) gives
	\begin{equation}\label{lemma_pre}
		\inj(M)=\operatorname{conj}(M)\ge\pi/\sqrt\kappa>\pi/\sqrt3.
	\end{equation}
	Fix \(p\in M\) and a unit-speed radial geodesic
	\(\gamma:[0,\pi/\sqrt3]\to M\) with \(\gamma(0)=p\).  In a parallel
	orthonormal trivialization of \(\dot\gamma^\perp\), let \(A\) be the Jacobi
	tensor
	\begin{equation}\label{def_A}
		A''+\cR_\gamma A=0,\qquad A(0)=0,\qquad A'(0)=\Id,
		\qquad \cR_\gamma(v)=R(v,\dot\gamma)\dot\gamma .
	\end{equation}
	For \(0<t<\pi/\sqrt3\), \(A(t)\) is nonsingular and \(S=A'A^{-1}\) is
	symmetric, satisfying the Riccati equation
	\begin{equation}\label{eq:Riccati}
		S'+S^2+\cR_\gamma=0,
	\end{equation}
	while the radial volume density is \(J(t,\theta)=\det A(t,\theta)\) with
	\begin{equation}\label{eq:logdet-H}
		\frac{\dd}{\dd t}\log\det A=\tr S.
	\end{equation}
	By the Hessian comparison and \eqref{lemma_pre}, we have 
	\begin{equation}\label{lem:S-bounds}
		c(t)\Id<S(t)<b(t)\Id,
		\qquad 0<t<\frac{\pi}{\sqrt3},
	\end{equation}
	with \(b(t)=1/t\) and \(c(t)=\sqrt3\cot(\sqrt3\,t)\). It follows that
	\[
	\det A(t)\ge
	\left(\frac{\sin(\sqrt3t)}{\sqrt3}\right)^3, 	\qquad 0<t<\frac{\pi}{\sqrt3}.
	\]
	By \eqref{lemma_pre}, we can conclude that
	\[
	\Vol(M,g)>
	2\pi^2\int_0^{\pi/\sqrt3}
	\left(\frac{\sin(\sqrt3t)}{\sqrt3}\right)^3\,dt
	=
	\frac{8\pi^2}{27}.
	\]
However, this lower bound for the volume is not sufficient for the proof of
Theorem~\ref{thm:main}. Notice that, when the metric is normalized so that
$\Ric=3g$,
\[
\Vol(S^4)=\frac{8\pi^2}{3}
\qquad\text{and}\qquad
\Vol(\CP^2)=2\pi^2,
\]
which suggests that a sharper estimate should be attainable.
	
	 Next, we obtain an improved estimate.

	\begin{proposition}\label{prop:candle}
		Suppose that \((M^4,g)\) is compact, orientable, with \(\Ric=3g\) and \(K>0\). Then we have
		\begin{equation}\label{eq:volume-floor}
			\Vol(M,g)>L:=\frac{2\pi^2(\pi^2-4)}9.
		\end{equation}
	\end{proposition}
	
	\begin{proof}
		Let the eigenvalues of \(S=A'A^{-1}\) be \(x_1,x_2,x_3\), where $A$ is defined in \eqref{def_A}.  By \eqref{lem:S-bounds},
		\((x_i-b)(x_i-c)<0\), i.e.
		\begin{equation}\label{eq:chord}
			x_i^2<(b+c)x_i-bc,
		\end{equation}
		and summing over \(i\),
		\begin{equation}\label{eq:trS2}
			\tr S^2<(b+c)\tr S-3bc,
		\end{equation}
		where \(b:=b(t)=1/t\) and \(c:=c(t)=\sqrt3\cot(\sqrt3\,t)\).
		Set
		\[
		H:=\tr S.
		\]
		Taking the trace of \eqref{eq:Riccati} and using the Einstein condition $
		\tr\cR_\gamma
		=
		\Ric(\dot\gamma,\dot\gamma)
		=
		3,
		$
		we obtain
		\[
		H'=-\tr S^2-3.
		\]
		Using \eqref{eq:trS2},
		\begin{equation}\label{eq:H-ineq}
			H'>-(b+c)H+3bc-3.
		\end{equation}
		Define
		\begin{equation}\label{eq:H0}
			H_0:=c+2b.
		\end{equation}
		A direct differentiation,
		\[
		b'=-b^2,
		\qquad
		c'=-(c^2+3),
		\]
		which implies that
		\begin{align*}
			H_0'
			&=
			-c^2-3-2b^2\\
			&=
			-(b+c)(c+2b)+3bc-3\\
			&=
			-(b+c)H_0+3bc-3.
		\end{align*}
		Subtracting this equality from \eqref{eq:H-ineq},
		\begin{equation}\label{eq:Hdiff}
			(H-H_0)'+(b+c)(H-H_0)>0.
		\end{equation}
		Since
		\[
		\left(\log[t\sin(\sqrt3\,t)]\right)'
		=
		\frac1t+\sqrt3\cot(\sqrt3 t)
		=
		b+c,
		\]
		\eqref{eq:Hdiff} is equivalent to
		\begin{equation}\label{eq:integrating-factor}
			\left[
			t\sin(\sqrt3t)\,(H-H_0)
			\right]'>0.
		\end{equation}
		Near \(t=0\), the Jacobi tensor has the expansion
		\[
		A(t)
		=
		t\Id-\frac{t^3}{6}\cR_\gamma(0)+O(t^4),
		\]
		and hence
		\[
		S(t)
		=
		\frac1t\Id-\frac{t}{3}\cR_\gamma(0)+O(t^2).
		\]
		Taking traces and using \(\tr\cR_\gamma(0)=3\),
		\[
		H(t)=\frac3t-t+O(t^2).
		\]
		On the other hand,
		\[
		c(t)
		=
		\frac1t-t+O(t^3),
		\]
		so
		\[
		H_0(t)=c+2b=\frac3t-t+O(t^3).
		\]
		Therefore
		\[
		t\sin(\sqrt3t)(H-H_0)\longrightarrow0
		\qquad(t\downarrow0).
		\]
		By \eqref{eq:integrating-factor},
		\begin{equation}\label{eq:H-lower}
			H>H_0
			=
			\sqrt3\cot(\sqrt3t)+\frac2t.
		\end{equation}
		
		Let \(J_0(t)=\frac{t^2\sin(\sqrt3t)}{\sqrt3}\).  Then
		\(\frac{\dd}{\dd t}\log J_0=H_0<\frac{\dd}{\dd t}\log\det A\) by
		\eqref{eq:H-lower} and \eqref{eq:logdet-H}, so \(\det A/J_0\) is strictly
		increasing; both functions equal \(t^3+O(t^5)\) near \(0\), so the ratio
		tends to \(1\) and hence 
		\begin{equation}
			\det A(t)>\frac{t^2\sin(\sqrt3t)}{\sqrt3},
			\qquad 0<t<\frac{\pi}{\sqrt3}.
		\end{equation}
		It follows from \eqref{lemma_pre} that
		\[
		\Vol(M)>
		\Vol(S^3)\int_0^{\pi/\sqrt3}\frac{t^2\sin(\sqrt3t)}{\sqrt3}\,\dd t
		=
		\frac{2\pi^2}{9}\int_0^\pi x^2\sin x\,\dd x
		=
		\frac{2\pi^2(\pi^2-4)}9,
		\]
		where \(x=\sqrt3t\) and \(\int_0^\pi x^2\sin x\,\dd x=\pi^2-4\).  	
	\end{proof}

	\section{Proof of Theorem \ref{thm:main} and Corollary \ref{thm:main_2}}

	\label{sec:rigidity}
	Now we give the proof of Theorem \ref{thm:main}.
	\begin{proof}[Proof of Theorem \ref{thm:main}]
		Notice that
		Synge's theorem  yields simple connectivity, whence
		\(H^1(M;\R)=H^3(M;\R)=0\), and Poincar\'{e} duality gives	
		\begin{equation}\label{eq:chi-b2}
			\chi(M)=2+b_2(M)>0.
		\end{equation}	
		Let \(V=\Vol(M,g)\).  \textbf{Under the normalization, we may assume \(\Ric=3g\)  and \(s=12\)}.
		Since $b_2(M) \le 1$ by the assumption and the signature $\tau=\chi\ (mod) \ 2$, we only need to consider the following two cases:
		
		\subsection*{Caese 1: \((\chi,\tau)=(2,0)\)}
		
		By \eqref{eq:signature-general}, $$\int|W^+|^2\dd V= \int|W^-|^2\dd V;$$ we write $F:=\int|W^+|^2\dd V= \int|W^-|^2\dd V$.
		By \eqref{eq:CGB-einstein},
		we have
		\begin{equation}
			F=8\pi^2-3V.
		\end{equation}
		If \(F>0\), Theorem~\ref{thm:GL-gap}  give
		\(F\ge6V\), hence \(V\le\frac{8\pi^2}{9}\), contradicting
		Proposition~\ref{prop:candle}.  Thus \(F=0\) and \(W^+=W^-=0\).  The metric is
		then conformally flat and Einstein, so \(\cR=\Id\) on \(\Lambda^2\) and \((M,g)\) is the unit round
		\(S^4\).
		
		\subsection*{Case 2: \((\chi,\tau)=(3,1)\) or  \((\chi,\tau)=(3,-1)\)} 
		
		We first consider the case \((\chi,\tau)=(3,1)\).
		Put \(A=\int|W^+|^2\dd V\) and \(B=\int|W^-|^2\dd V\).  By
		\eqref{eq:signature-general} and \eqref{eq:CGB-einstein},
		\begin{equation}\label{eq:A-B}
			A-B=12\pi^2,
			\qquad
			A+B+6V=24\pi^2,
		\end{equation}
		so
		\begin{equation}\label{eq:B-formula}
			B=6\pi^2-3V.
		\end{equation}
		If \(B>0\), the gap estimate gives \(B\ge6V\), hence \(V\le\frac{2\pi^2}{3}\),
		again contradicting Proposition~\ref{prop:candle}.  Therefore \(B=0\),
		\(W^-=0\).
		Since \(W^-\equiv0\), the metric \(g\) is self-dual.
		Hitchin's classification theorem
		\cite[Theorem~13.30]{Besse}
		therefore implies that \((M,g)\) is, up to homothety,
		isometric to the round four-sphere or to
		\(\mathbb{CP}^2\) with its Fubini--Study metric.
		The round case is excluded by \(\tau(M)=1\).
		Hence
		\[
		(M,g)\cong(\mathbb{CP}^2,g_{\mathrm{FS}}).
		\]
		
		If  \((\chi,\tau)=(3,-1)\), reverse the
		orientation; this interchanges \(W^+\) and \(W^-\) and the preceding argument
		applies verbatim.	
		This proves Theorem~\ref{thm:main}.
	\end{proof}

		Next, we give the proof of  Corollary \ref{thm:main_2}.
	
		\begin{proof}[Proof of  Corollary \ref{thm:main_2}]
			 By the theorem of Hsiang
			and Kleiner \cite{HsiangKleiner}, a compact, orientable, positively
			curved four-manifold with a nontrivial Killing vector field satisfies $\chi\le 3$. Then Corollary \ref{thm:main_2} follows from
			Theorem~\ref{thm:main} directly.
			\end{proof}

\bigskip
\noindent\textbf{Acknowledgments.} The author would like to thank Professor Xiaochun Rong for many inspiring discussions on this work.

\end{document}